\documentclass{amsart}
\usepackage{amsmath}
\usepackage{amssymb}
\usepackage{mathrsfs}
\usepackage{pifont}
\usepackage{amsfonts}
\usepackage{stmaryrd}
\usepackage{latexsym,amssymb,amsmath}

\newtheorem{theorem}{Theorem}[section]
\newtheorem{lemma}[theorem]{Lemma}
\newtheorem{corollary}[theorem]{Corollary}

\theoremstyle{definition}

\theoremstyle{remark}
\newtheorem{remark}[theorem]{Remark}

\numberwithin{equation}{section}

\begin{document}

\title{Generalized Verma Modules and Character Formulae for $\mathfrak{osp}(3|2m)$}

\author{Bintao Cao}
\address{Academy of Mathematics and Systems Science, Chinese Academy of Sciences, Beijing, 100190, China}
\email{caobintao@amss.ac.cn}

\author{Li Luo}
\address{Department of Mathematics, East China Normal University, Shanghai, 200241, China}
\email{lluo@math.ecnu.edu.cn}





\begin{abstract}

The character formula of any finite dimensional irreducible module
for Lie superalgebra $\mathfrak{osp}(3|2m)$ is obtained in terms of characters of generalized Verma modules.

\end{abstract}

\maketitle



\section{Introduction}

\subsection{}
It was already clear from the foundation papers
\cite{K1,K2,K3} by Kac that finding the character formulae of the
so-called atypical irreducible modules is one of the most
challenging problems in the theory of Lie superalgebras.

The finite-dimensional irreducible modules of the Lie superalgebras
$D(2,1;\alpha)$ and $C(n)$ were understood by van der Jeugt in
\cite{V1,V2} long ago. In the case of $\mathfrak{gl}(m|n)$, the
problem was first solved by Serganova in \cite{Se1}, where a
Kazhdan-Lusztig theory was developed. Later on, Brundan \cite{B}
discovered a remarkable connection between the general linear
superalgebra and the quantum group of $\mathfrak{gl}_\infty$, by
which the Kazhdan-Lusztig polynomials were computed quite directly.
This approach was further developed in \cite{SZ1, CWZ}.

Very recently, Gruson and Serganova \cite{GS}, and, Cheng, Lam and
Wang \cite{CLW} provided two independent and different approaches to
the finite-dimensional irreducible $\mathfrak{osp}$-characters.

\subsection{}
At the same time, the study of generalized Verma modules for
ortho-symplectic Lie superalgebra $\mathfrak{osp}(n|2)$ has recently
enjoyed considerable attention in \cite{L} by Luo and in \cite{SZ2}
by Su and Zhang. This effort is not only possibly useful to describe
the finite-dimensional irreducible $\mathfrak{osp}$-characters by
another approach completely different from \cite{GS, CLW} but also a
first step toward studying category $\mathcal {O}$.

Choose the distinguished Borel subalgebra for a basic classical Lie
superalgebra and consider the maximal parabolic subalgebra obtained
by removing the unique odd simple root. The generalized Verma
modules are those induced from finite dimensional irreducible Verma
modules over the parabolic subalgebra. In the case of types $A$ and
$C$, these modules coincide the Kac modules, which can be
interpreted as cohomology groups of bundles over the flag
supermanifold. However, in the case of types $B$ and $D$, the Kac
modules are not the generalized Verma modules but their maximal
finite-dimensional quotients. This present paper, as well as the
papers \cite{L, SZ2}, implies that the generalized Verma modules
also play important roles in the representation theory for Lie
superalgebras.

\subsection{}
Influenced by the work \cite{L}, here we shall also study any
finite-dimensional irreducible module $L_\lambda$ from $L_\mu\otimes
\mathbb{C}^{3|2m}$ where $\mathbb{C}^{3|2m}$ is the natural
representation and $L_\mu$ is another irreducible module known by
induction. We point out that the module $L_\mu\otimes
\mathbb{C}^{3|2m}$ may not be completely reducible. Thus we should
give a more detailed argument for the blocks of modules appearing in
$L_\mu\otimes \mathbb{C}^{3|2m}$.

Another key point of this work is how to express the trivial
character $\mbox{ch } L_0=1$ in terms of characters of generalized
Verma modules, for which a nontrivial combinatorial identity is
proved. This combinatorial identity is formulated by a certain
subset of Weyl group.

\subsection{}
The purpose of the present paper is to express the character
formulae for all finite dimensional irreducible modules of
$\mathfrak{osp}(3|2m)$ in terms of characters of generalized Verma
modules, which are known clearly. Our main results are Theorem 3.1
and Theorem 5.2.

The paper is organized as follows. In Section 2, we present some
background material on $\mathfrak{osp}(3|2m)$. In Section 3, the
character formulae for tail atypical weights are obtained. Section 4
is devoted to study the information of tensor modules. The result
for non-tail atypical weight is given in Section 5.

\section{Preliminaries}

We shall simplify $\mathfrak{osp}(3|2m)$ to $\mathfrak{g}$ and work
over the field $\mathbb{C}$ of complex numbers throughout the paper.
\subsection{Distinguished simple roots system} Choose the
distinguished Borel subalgebra \cite{K1} for $\mathfrak{g}$. Then
the set of simple roots is
\begin{equation}
\Pi=\{\delta_1-\delta_2,\delta_2-\delta_3,\ldots,\delta_{m-1}-\delta_m,\delta_m-\epsilon,\epsilon\}.
\end{equation}

The set of positive even roots and odd roots are
\begin{equation}
\Delta_0^+=\{\epsilon, \delta_i\pm\delta_j,2\delta_k\mid 1\leq
i<j\leq m,1\leq k\leq m\}
\end{equation}
and
\begin{equation}
\Delta_1^+=\{\delta_i,\delta_i\pm\epsilon\mid 1\leq i\leq m\},
\end{equation}
respectively.

The distinguished Dynkin diagram is given as follows:

\vspace{0.5cm} \setlength{\unitlength}{3pt}
\begin{picture}(108,6)\put(0,-1){$B(1|m)$:}
\put(22,0){\circle{2}}\put(17,-5){$\delta_1-\delta_2$}\put(23,0){\line(1,0){12}}
\put(36,0){\circle{2}}\put(32,-5){$\delta_2-\delta_3$}
\put(42,-0.5){$\cdots$}\put(52,0){\circle{2}}\put(53,0){\line(1,0){12}}
\put(58,-5){$\delta_m-\epsilon$} \put(44,-12){{\tt(Figure 1)}}
\put(65,-1){$\otimes$}\put(67,-0.25){\line(1,0){12}}\put(67,0.25){\line(1,0){12}}\put(80,0){\circle{2}}
\put(73,1.75){\line(3,-1){5.5}}\put(73,-1.75){\line(3,1){5.5}}\put(80,-5){$\epsilon$}
\end{picture}
\vspace{1.5cm}

The bilinear form $(\cdot,\cdot)$ on $\mathfrak{h}^*$ is defined by
\begin{equation}
(\delta_i,\delta_j)=\left\{\begin{array}{cc}-1,& (i=j)\\ 0, &(i\neq
j)\end{array}\right., \quad (\delta_i,\epsilon)=0,\quad
(\epsilon,\epsilon)=1 \quad \mbox{for}\quad 1\leq i,j\leq m.
\end{equation}

\subsection{$\mathbb{Z}$-grading and parabolic subalgebras}

The Lie superalgebra $\mathfrak{g}$ admits a
$\mathbb{Z}_2$-consistent $\mathbb{Z}$-grading
\begin{equation}
\mathfrak{g}=\mathfrak{g}_{-2}\oplus\mathfrak{g}_{-1}\oplus\mathfrak{g}_0\oplus\mathfrak{g}_{1}\oplus\mathfrak{g}_{2},
\end{equation} where
$\mathfrak{g}_0\cong\mathfrak{gl}(m)\oplus\mathfrak{sl}(2)$ is
spanned by $\mathfrak{h}$, $\mathfrak{g}_{\pm\epsilon}$ and
$\mathfrak{g}_{\pm(\delta_i-\delta_j)}$ $(1\leq i<j \leq m)$;
$\mathfrak{g}_{\pm1}$ is spanned by $\mathfrak{g}_{\pm\alpha}
(\alpha\in\Delta_1^+)$; and $\mathfrak{g}_{\pm2}$ is spanned by
$\mathfrak{g}_{\pm(\delta_i+\delta_j)}$ $(1\leq i\leq j \leq m)$.
There are two parabolic subalgebras
\begin{equation}
\mathfrak{u}=\mathfrak{g}_{1}\oplus\mathfrak{g}_{2}\quad
\mbox{with}\quad \mathfrak{u}_{\bar{0}}=\mathfrak{g}_{2} \mbox{ and
} \mathfrak{u}_{\bar{1}}=\mathfrak{g}_{1}
\end{equation} and
\begin{equation}
\mathfrak{u}^-=\mathfrak{g}_{-1}\oplus\mathfrak{g}_{-2}\quad
\mbox{with}\quad \mathfrak{u}^-_{\bar{0}}=\mathfrak{g}_{-} \mbox{
and } \mathfrak{u}^-_{\bar{1}}=\mathfrak{g}_{-1}.
\end{equation}

\subsection{Dominant integral weights} An element in $\mathfrak{h}^*$ is called a {\em weight}.
A weight $\lambda\in\mathfrak{h}^*$ will be written in terms of the
$\delta\epsilon$-basis as
\begin{equation}
\lambda=(\lambda_1,\lambda_2,\ldots,\lambda_m;\lambda_0)=\sum_{i=1}^m
\lambda_i\delta_i+\lambda_0\epsilon.
\end{equation}
Define
\begin{equation}
\mbox{ht }\lambda=\sum_{i=0}^m |\lambda_i|,
\end{equation}
and call it the height of $\lambda$.

A weight $\lambda$ is called {\em integral} if
$\lambda_0,\lambda_1,\ldots,\lambda_m\in\mathbb{Z}$.

For any weight $\lambda\in\mathfrak{h}^*$, there exists an
irreducible module $L_\lambda$ with highest weight $\lambda$.

\begin{theorem}{\bf(Kac \cite{K1})} For any integral weight $\lambda$, the irreducible
module $L_{\lambda}$ is finite dimensional if and only if
\begin{equation}
\left\{\begin{array}{l}  \lambda_0, \lambda_1,\ldots,\lambda_{m}\in
\mathbb{Z}_{\geq0},\\
\lambda_1\geq \lambda_2\geq\cdots\geq \lambda_{m-1}\geq\lambda_m,\\
\lambda_0=0\quad \mbox{if}\quad \lambda_m=0.
\end{array}\right.
\end{equation}
\end{theorem}

\vspace{0.3cm} Denote
\begin{equation}
\mathcal {P}=\{\lambda\in\mathfrak{h}^*\mid \mbox{$\lambda$ is
integral and satisfies the conditions in Theorem 2.1}\}.
\end{equation}
The weights in $\mathcal{P}$ are called {\em dominant integral}.

\subsection{Generalized Verma modules}
For any $\lambda\in \mathfrak{h}^*$, we denote by $L^{(0)}_\lambda$
the irreducible $\mathfrak{g}_0$- module with highest weight
$\lambda$. Denote
\begin{equation}
\mathcal {P}^{(0)}=\{\lambda\in\mathfrak{h}^*\mid
\mbox{$L_\lambda^{(0)}$ is finite-dimensional}\}.
\end{equation}
It is obvious that $\mathcal {P}\subset \mathcal {P}^{(0)}$.

For any $\lambda\in \mathcal {P}^{(0)}$, extend $L^{(0)}_\lambda$ to
a $\mathfrak{g}_0\oplus\mathfrak{u}$-module by putting
$\mathfrak{u}L^{(0)}_\lambda=0$. Then the {\em generalized Verma
module} $M_\lambda$ is defined as the induced module
\begin{equation}
M_\lambda=\mbox{Ind}^{\mathfrak{g}}_{\mathfrak{g}_0\oplus\mathfrak{u}}L^{(0)}_\lambda\cong
{U}(\mathfrak{u}^-)\otimes_{\mathbb{C}}L^{(0)}_\lambda.
\end{equation}
It is clear that $L_\lambda$ is the unique irreducible quotient
module of $M_\lambda$.

\subsection{Atypical weights}
Let $\rho_0$ (resp. $\rho_1$) be half the sum of positive even
(resp. odd) roots, and let $\rho=\rho_0-\rho_1$. Then
\begin{equation}
\rho=(m-\frac{3}{2})\delta_1+(m-\frac{5}{2})\delta_2+\cdots+
\frac{1}{2}\delta_{m-1}-\frac{1}{2}\delta_m+\frac{1}{2}\epsilon.
\end{equation}

A weight $\lambda\in \mathfrak{h}^*$ is called {\em atypical} if
there is a positive odd root $\alpha\in\bar{\Delta}_1^+$ such that
\begin{equation}
(\lambda+\rho,\alpha)=0.
\end{equation} Sometimes we also call it {\em $\alpha$-atypical} to emphasize the odd root $\alpha$. Otherwise, we call $\lambda$ {\em
typical}.

Note that if $\lambda\in\mathcal {P}$ is
$(\delta_t-\epsilon)$-atypical for some $1\leq t\leq m$, then it
must be that $t=m$ and
$\lambda_1\geq\lambda_2\geq\cdots\geq\lambda_{m-1}\geq\lambda_m=\lambda_0=0$.
Such a weight $\lambda$ is called a \emph{tail atypical weight}. A
weight $\lambda\in\mathcal {P}$ which is not a tail atypical weight,
is called a \emph{non-tail atypical weight}.

For convenience, sometimes we use the \emph{$\rho$-translated}
notation $\widetilde{\lambda}$ of a weight $\lambda$, which is
defined by
\begin{equation}
\widetilde{\lambda}=\lambda+\rho=(\widetilde{\lambda}_1,\widetilde{\lambda}_2,\ldots,\widetilde{\lambda}_m;\widetilde{\lambda}_0)
\end{equation}

Assume $\lambda\in \mathfrak{h}^*$ is $\alpha$-atypical with
$\alpha=\delta_t+\epsilon$ or $\delta_t-\epsilon$. Denote
\begin{equation}
\overline{\lambda}=[|\widetilde{\lambda}_1|,\ldots,|\widetilde{\lambda}_{t-1}|,
|\widetilde{\lambda}_{t+1}|,\ldots,|\widetilde{\lambda}_m|],
\end{equation} which is called the \emph{atypical type} of $\lambda$.

\subsection{Weyl group} Let $\mathcal {W}$ be the Weyl group of
$\mathfrak{g}$ (i.e. the Weyl group of Lie algebra
$\mathfrak{g}_{\bar{0}}$) and $\ell: \mathcal {W}\rightarrow
\mathbb{N}$ be the length function. Let $S_{k}\ltimes
\mathbb{Z}_2^{k}\subset\mathcal {W}$ be the Weyl group of type $C_k$
($1\leq k\leq m$). That is, elements in $S_k$ permute the
coefficients of $\delta_1,\ldots,\delta_k$ and elements in
$\mathbb{Z}_2^{k}$ show how to change the signs of the coefficients
of $\delta_1, \ldots,\delta_k$.

For any $1\leq k\leq m$, denote
\begin{equation}\label{def}
\Gamma_k:=\{\sigma\in S_{k}\ltimes
\mathbb{Z}_2^{k}\mid
\sigma(\mu)_1>\sigma(\mu)_2>\cdots>\sigma(\mu)_{k}\},
\end{equation} where $\mu=k\delta_1+(k-1)\delta_2+\cdots+\delta_k$ and $\sigma(\mu)_i$ is the
coefficient of $\delta_i$ in $\sigma(\mu)$.

\begin{remark}\label{r3.1}
  Note that $\sigma=\omega\cdot\varepsilon\in\Gamma_{m-1}$ is determined by $\varepsilon$ uniquely,
  where $\omega\in S_{m-1}$ and $\varepsilon\in\mathbb{Z}_2^{m-1}$.
  Hence there is a 1-1 correspondence between the elements in $\Gamma_{m-1}$ and the self-conjugate
  partitions with length smaller than $m$ via
  \begin{equation}\label{r3.2}
    \mu=(j_p-1,\ldots,j_1-1|j_p-1,\ldots,j_1-1)\longmapsto\sigma_{\mu}=\omega_{\mu}\cdot\varepsilon_{\mu},
  \end{equation} where $1\leq j_1\leq\cdots\leq m-1$ and $\varepsilon_{\mu}=(\varepsilon_{\mu,1},\ldots,\varepsilon_{\mu,m-1})$
   with \begin{eqnarray}\label{r3.3}
    \varepsilon_{\mu,s}=\left\{\begin{array}{ll}
      -1& \mbox{if}\ \  s\in\{m-j_1,\ldots,m-j_p\}\\
      1& \mbox{otherwise}.
    \end{array}\right.
  \end{eqnarray}
\end{remark}

\subsection{Central characters and blocks}
Denote by $Z(\mathfrak{g})$ the central of the universal enveloping
algebra $U(\mathfrak{g})$. Recall that a \emph{central character}
defined by $\lambda\in\mathfrak{h}^*$ is a homomorphism
$\chi_\lambda: Z(\mathfrak{g})\rightarrow\mathbb{C}$ such that each
element $z\in Z(\mathfrak{g})$ acts on $L_\lambda$ as the scaler
$\chi_\lambda(z)$.
\begin{lemma} {\em (See. \cite{Se2})}
For any $\lambda,\mu\in\mathcal {P}^{(0)}$, $\chi_\lambda=\chi_\mu$
implies that $\overline{\lambda}=\overline{\mu}$ if $\lambda,\mu$
are both atypical or that there exists $w\in\mathcal {W}$ such that
$\widetilde{\lambda}=w(\widetilde{\mu})$ if $\lambda,\mu$ are both
typical.
\end{lemma}

Suppose $\lambda\in\mathcal {P}$ is typical, then for any
$\mu\in\mathcal {P}$, $\chi_\lambda=\chi_\mu$ if and only if
$\lambda=\mu$.

For any atypical weight $\lambda\in\mathcal {P}$, we denote by $\lambda^T$ the
unique tail atypical weight with
$\chi_\lambda=\chi_{\lambda^T}$. It is obvious that $\lambda=\lambda^T$ if $\lambda$ is a tail atypical weight. For a non-tail
atypical weight $\lambda\in\mathcal {P}$, one can check that
\begin{equation}
\lambda^T=\sum_{i=1}^{k-1}\lambda_i\delta_i+\sum_{i=k}^{m-1}\lambda_{i+1}\delta_i\quad\mbox{if
$\lambda$ is $(\delta_k+\epsilon)$-atypical.}
\end{equation}

For any tail atypical weight $\lambda$, denote
\begin{equation}
B_\lambda=\{\mu\in\mathcal {P}\mid \mu^T=\lambda\}.
\end{equation}
Clearly,
\begin{equation}
\mathcal {P}=\bigcup_{\lambda\mbox{\ tail atypical}}B_\lambda,
\quad\mbox{which is a disjoint union.}
\end{equation}

Take any non-tail $(\delta_k+\epsilon)$-atypical weight $\lambda$
$(1\leq k\leq m)$ and assume that $t$ is the maximal number with
$\lambda_t=\lambda_k$. Define
\begin{equation}
\varphi(\lambda)=\left\{\begin{array}{ll}\lambda-\delta_m=\lambda^T,&\quad\mbox{if
$\lambda_0=0$}\\\lambda-\sum_{i=k}^t\delta_i-(t-k+1)\epsilon,
&\quad\mbox{otherwise}.
\end{array}\right.
\end{equation}

It is obvious that $\varphi: \mathcal {P}\setminus\{\mbox{tail
atypical weights}\}\rightarrow\mathcal {P}$ is injective. Moreover,
one can check that $\varphi(\mu)\in B_\lambda$ for any tail typical
weight $\lambda$ and any $\mu\in B_\lambda$.

\subsection{Characters for typical weights}
\begin{theorem}\label{T2.8.1}\em{(Kac\cite{K2})}
A weight $\lambda\in\mathcal {P}$ is typical if and only if
\begin{equation}\label{2.8.1}
ch\
L_{\lambda}=\frac{\prod_{\alpha\in\Delta_1^+}(e^{\frac{\alpha}{2}}+e^{-\frac{\alpha}{2}})}
{\prod_{\alpha\in\Delta_0^+}(e^{\frac{\alpha}{2}}-e^{-\frac{\alpha}{2}})}
\sum_{w\in \mathcal {W}}(-1)^{l(w)}e^{w(\lambda+\rho)}.
\end{equation}

We rewrite this theorem as the following lemma.
\end{theorem}
\begin{lemma}\label{T2.8.2}
A weight $\lambda\in\mathcal {P}$ is typical if and only if
\begin{equation}\label{2.8.2}
\mbox{ch }L_\lambda=\sum_{\sigma\in\Gamma_m}(-1)^{\ell(\sigma)}\mbox{ch
}M_{\sigma(\lambda+\rho)-\rho}.
\end{equation}
\end{lemma}
\begin{proof}
By Theorem \ref{T2.8.1}, one has that $\lambda\in\mathcal {P}$ is typical if and only if
\begin{eqnarray}\label{2.8.3}
ch\
L_{\lambda}=\frac{\prod\limits_{\alpha\in\Delta_1^+}(1+e^{-\alpha})}{\prod\limits_{1\leq
i\leq j\leq m}(1-e^{-(\delta_i+\delta_j)})} \frac{\sum\limits_{w\in
\mathcal
{W}}(-1)^{l(w)}e^{w(\lambda+\rho)-\rho}}{(1-e^{-\epsilon})\prod\limits_{1\leq
i<j\leq m}(1-e^{-(\delta_i-\delta_j)})}.
\end{eqnarray}

For convenience, we denote $\mu'=\sum_{j=1}^m\mu_j\delta_j$ with
$\mu=\sum_{j=1}^m\mu_j\delta_j+\mu_0\epsilon$, and denote
$\rho_{\mathfrak{gl}(m)}=\sum_{j=1}^m(m-j)\delta_j$ being half the
sum of positive roots of $\mathfrak{sl}(m)$. Since $\mathcal
{W}=\mathcal {W}_{C_m}\times\mathcal {W}_{B_1}$, we have
\begin{eqnarray}\label{2.8.4}\\\nonumber
\sum\limits_{w\in \mathcal {W}}(-1)^{l(w)}e^{w(\lambda+\rho)-\rho}=
(e^{\lambda_0\epsilon}-e^{-(\lambda_0+1)\epsilon})\sum\limits_{w\in
S_m\ltimes\mathbb{Z}_2^m}(-1)^{l(w)}e^{w(\lambda'+\rho')-\rho'}.
\end{eqnarray} Thus
\begin{eqnarray*}
 ch\ L_{\lambda}&=&\frac{\prod\limits_{\alpha\in\Delta_1^+}(1+e^{-\alpha})}{\prod\limits_{1\leq
i\leq j\leq
m}(1-e^{-(\delta_i+\delta_j)})}\sum_{j=0}^{2\lambda_0}e^{(\lambda_0-j)\epsilon}
\frac{\sum\limits_{w\in
S_m\ltimes\mathbb{Z}_2^m}(-1)^{l(w)}e^{w(\lambda'+\rho')-\rho'}}{\prod\limits_{1\leq
i<j\leq m}(1-e^{-(\delta_i-\delta_j)})}\\
&=&\frac{\prod\limits_{\alpha\in\Delta_1^+}(1+e^{-\alpha})}{\prod\limits_{1\leq
i\leq j\leq
m}(1-e^{-(\delta_i+\delta_j)})}\sum_{j=0}^{2\lambda_0}e^{(\lambda_0-j)\epsilon}
\sum_{\sigma\in\Gamma_m}(-1)^{l(\sigma)}\frac{\sum\limits_{w\in
S_m}(-1)^{l(w)}e^{w(\sigma(\lambda'+\rho')-\rho'+\rho_{\mathfrak{gl}(m)})}}{e^{\rho_{\mathfrak{gl}(m)}}\prod\limits_{1\leq
i<j\leq m}(1-e^{-(\delta_i-\delta_j)})}\\&=&
\frac{\prod\limits_{\alpha\in\Delta_1^+}(1+e^{-\alpha})}{\prod\limits_{1\leq
i\leq j\leq
m}(1-e^{-(\delta_i+\delta_j)})}\sum_{\sigma\in\Gamma_m}(-1)^{l(\sigma)}ch\
L^{(0)}_{\sigma(\lambda+\rho)-\rho}\\&=&
\sum_{\sigma\in\Gamma_m}(-1)^{l(\sigma)}ch\
M_{\sigma(\lambda+\rho)-\rho}
\end{eqnarray*}
\end{proof}

\subsection{Cohomology and character}
The space of $q$-dimensional cochains of the Lie superalgebra
$\mathfrak{u}=\mathfrak{u}_{\bar{0}}\otimes\mathfrak{u}_{\bar{1}}$
with coefficients in the module $L_\lambda$ is given by
\begin{equation}
C^q(\mathfrak{u};
L_\lambda)=\bigoplus_{q_0+q_1=q}\mbox{Hom}(\wedge^{q_0}\mathfrak{u}_{\bar{0}}\otimes
S^{q_1}\mathfrak{u}_{\bar{1}}, L_\lambda).
\end{equation}
The differential $d: C^q(\mathfrak{u}; L_\lambda)\rightarrow
C^{q+1}(\mathfrak{u}; L_\lambda)$ is defined by
\begin{equation}
\begin{array}{c}
d c(\xi_1,\ldots,\xi_{q_0},\eta_1,\ldots,\eta_{q_1})\\=\sum_{1\leq s
\leq t\leq
q_0}(-1)^{s+t-1}c([\xi_s,\xi_t],\xi_1,\ldots,\widehat{\xi_s},\ldots,\widehat{\xi_t},\ldots,
\xi_{q_0},\eta_1,\ldots,\eta_{q_1})
\\
+\sum_{s=1}^{q_0}\sum_{t=1}^{q_1}(-1)^{s-1}c(\xi_1,\ldots,\widehat{\xi_s},\ldots,\xi_{q_0},
[\xi_s,\eta_t],\eta_1,\ldots,\widehat{\eta_t},\ldots,\eta_{q_1})\\
+\sum_{1\leq s\leq t\leq
q_1}c([\eta_s,\eta_t],\xi_1,\ldots,\xi_{q_0},
\eta_1,\ldots,\widehat{\eta_s},\ldots,\widehat{\eta_t},\ldots,\eta_{q_1})\\
+\sum_{s=1}^{q_0}(-1)^s\xi_s c(\xi_1,\ldots,\widehat{\xi_s},\ldots,\xi_{q_0},\eta_1,\ldots,\eta_{q_1})\\
+(-1)^{q_0-1}\sum_{s=1}^{q_1}\eta_s
c(\xi_1,\ldots,\xi_{q_0},\eta_1,\ldots,\widehat{\eta_s},\ldots,\eta_{q_1})
\end{array}\end{equation}
where $c\in C^{q}(\mathfrak{u}; L_\lambda)$,
$\xi_1,\ldots,\xi_{q_0}\in\mathfrak{u}_{\bar{0}}$,
$\eta_1,\ldots,\eta_{q_1}\in\mathfrak{u}_{\bar{1}}$.

The cohomology of $\mathfrak{u}$ with coefficients in the module
$L_\lambda$ is the cohomology group of the complex
$C=(\{C^q(\mathfrak{u}; L_\lambda)\},d)$, and it is denoted by
$H^q(\mathfrak{u}, L_\lambda)$.

Imitating the argument in \cite{L}, one can obtain
\begin{equation}
\mbox{ch } L_\lambda=\mbox{ch }
M_\lambda+\sum_{\nu\prec\lambda}\sum_{i=0}^\infty(-1)^i[H^i(\mathfrak{u},
L_\lambda): L_\nu^{(0)}]\mbox{ch } M_\nu,
\end{equation} where $\nu\prec\lambda$ means that $\chi_\lambda=\chi_\nu$ and $\nu<\lambda$ (i.e. $\lambda-\mu$ is a
$\mathbb{Z}_{\geq 0}$-linear sum of positive roots), and
$[H^i(\mathfrak{u}, L_\lambda): L_\nu^{(0)}]$ is the multiplicity of
$L_\nu^{(0)}$ in the cohomology group regarded as a
$\mathfrak{g}_0$-module.

\begin{remark}
In the above expression, any coefficient of $\mbox{ch } M_\nu$ is an
integral number. In particular, the coefficient of $\mbox{ch }
M_\lambda$ is $1$.
\end{remark}

\section{Character formulae for tail atypical weights}

\subsection{Natural module}
Consider the natural $\mathfrak{g}$-module $L_{\delta_1}\simeq
\mathbb{C}^{3|2m}$. It is known that ${L_{\delta_1}}^*\simeq
L_{\delta_1}$ and that the set of weights of $L_{\delta_1}$ is $\{0,
\pm\delta_1,\ldots,\pm\delta_m,\pm\epsilon\}$. Moreover, as a
$\mathfrak{g}_0$-module,
\begin{equation}
L_{\delta_1}\simeq L_{\delta_1}^{(0)}\oplus
L_{-\delta_m}^{(0)}\oplus L_\epsilon^{(0)}.
\end{equation}

\subsection{Weights set $\mathcal {P}_\lambda$ and $\mathcal {P}_\lambda^{(0)}$}
For any $\lambda\in\mathcal {P}$, denote
\begin{equation}
\mathcal{P}_\lambda=\mathcal{P}_{\lambda^+}\cup\mathcal{P}_{\lambda^-}
\end{equation}
where
\begin{equation}
\mathcal{P}_{\lambda^+}=\left\{\begin{array}{ll}
\{\lambda+\delta_1,\lambda+\delta_{2},\ldots,\lambda+\delta_m,
\lambda+\epsilon\}\cap
\mathcal{P}, & (\lambda_0=0);\\
\{\lambda+\delta_1,\lambda+\delta_{2},\ldots,\lambda+\delta_m,
\lambda, \lambda+\epsilon\}\cap \mathcal{P}, &
(\lambda_0\neq0).\end{array}\right.
\end{equation} and
\begin{equation}
\mathcal{P}_{\lambda^-}=\left\{\begin{array}{ll}
\{\lambda-\delta_1,\lambda-\delta_{2},\ldots,\lambda-\delta_m,
\lambda-\epsilon\}\cap
\mathcal{P}, & (\lambda_0=0);\\
\{\lambda-\delta_1,\lambda-\delta_{2},\ldots,\lambda-\delta_m,
\lambda, \lambda-\epsilon\}\cap \mathcal{P}, &
(\lambda_0\neq0).\end{array}\right.
\end{equation}

Let $\mathcal {P}^{(0)}$ take the place of $\mathcal {P}$ above,
then there come the definitions of ${P}_\lambda^{(0)}$,
$\mathcal{P}_{\lambda^+}^{(0)}$ and $\mathcal{P}_{\lambda^+}^{(0)}$.

The following statement is standard in the theory of the classical
finite-dimensional semisimple Lie algebras: for any
$\lambda\in\mathcal {P}^{(0)}$,
\begin{equation}
L_\lambda^{(0)}\otimes (L_{\delta_1}^{(0)}\oplus
L_{-\delta_m}^{(0)}\oplus
L_\epsilon^{(0)})=\bigoplus_{\mu\in{P}_\lambda^{(0)}} L_\mu^{(0)}.
\end{equation}

\subsection{Character of $M_\lambda\otimes L_{\delta_1}$}
Thanks to (3.1) and (3.5), we have that for any $\lambda\in\mathcal
{P}^{(0)}$, as a $\mathfrak{g}_0$-module,
\begin{eqnarray}
\quad\quad\quad M_\lambda\otimes
L_\delta&=&({U}(\mathfrak{u}^-)\otimes_{\mathbb{C}}
L^{(0)}_\lambda)\otimes (L_{\delta_1}^{(0)}\oplus
L_{-\delta_m}^{(0)}\oplus
L_\epsilon^{(0)})\\\nonumber&\cong&{U}(\mathfrak{u}^-)\otimes_{\mathbb{C}}(L^{(0)}_\lambda\otimes
(L_{\delta_1}^{(0)}\oplus L_{-\delta_m}^{(0)}\oplus
L_\epsilon^{(0)}))\\\nonumber
&=&{U}(\mathfrak{u}^-)\otimes_{\mathbb{C}}(\bigoplus_{\mu\in\mathcal
{P}_\lambda^{(0)}}L^{(0)}_\mu)\\\nonumber&=&\bigoplus_{\mu\in\mathcal
{P}_\lambda^{(0)}} ({U}(\mathfrak{u}^-)\otimes_{\mathbb{C}}
L^{(0)}_\mu)\\\nonumber &=& \bigoplus_{\mu\in\mathcal
{P}_\lambda^{(0)}} M_{\mu}.
\end{eqnarray}
Hence
\begin{equation}
\mbox{ch }M_\lambda\otimes L_\delta=\sum_{\mu\in\mathcal
{P}_\lambda^{(0)}}\mbox{ch } M_{\mu}.
\end{equation}

\subsection{Convention} From here on, we will always simplify the
character $\mbox{ch }V$ to $V$. It cannot confuse us by context.

\subsection{Character formulae}Firstly, we will recall some notations and
facts about symmetric functions. One can find more material in
\cite{M}.

Denote $A_{\lambda}=\det(x_j^{\lambda_i})_{1\leq i,j\leq m}$ for
$\lambda=(\lambda_1,\ldots,\lambda_m)\in\mathbb{Z}^m$ and
$S_{\lambda}=A_{\lambda+\rho_{\mathfrak{gl}(m)}}/A_{\rho_{\mathfrak{gl}(m)}}$,
where $\rho_{\mathfrak{gl}(m)}=\sum_{j=1}^m(m-j)\delta_j$. Note that
$\lambda$ is not necessary a partition, and if so then we will get
back to the traditional definition of Schur function $S_{\lambda}$.
Thus $S_{\lambda}\neq0$ if and only if there exists $w\in S_m$
satisfying $\lambda_{\sigma(1)}-\sigma(1)+1\geq
\cdots\geq\lambda_{\sigma(m)}-\sigma(m)+m$. Equivalently,
$S_{\lambda}=0$ if and only if there exist $i,j\in\{1,\ldots,m\}$
satisfying that $\lambda_i-i=\lambda_j-j$.

A direct calculation shows that $S_{\lambda}=\det
(h_{\lambda_i+l-i+j})/e_m^l$, where $l\in\mathbb{Z}_{\geq0}$ is any nonnegative integer such that
$\lambda_i+l-i+m\geq0$ for all $i\in\{1,\ldots,m\}$. Here
$h_l=h_l(x_1,\ldots,x_m)$ is the $l$-th complete symmetric
polynomials (i.e. the sum of all monomials of total degree $l$ in
the variables $x_1,\ldots,x_m$), and $e_l=e_l(x_1,\ldots,x_m)$ is the
$l$-th elementary symmetric polynomials in the variables
$x_1,\ldots,x_m$. We set $h_l=e_l=0$ if $l<0$ for convenience.

For each $\sigma\in \Gamma_{m-1}$
(cf. \eqref{def}), define
\begin{equation}\label{3.1}
\flat(\sigma,\lambda)=\left\{\begin{array}{ll} i, & \mbox{if
$i\in\{1,2,\ldots,m-1\}$ is the minimal number such that
$\sigma(\lambda+\rho)_i<-1$};\\ m, & \mbox{if
$\sigma(\lambda+\rho)_{m-1}>-1$}.
\end{array}\right.
\end{equation}
Write
\begin{equation}\label{3.2}
\tau_i:=(i,i+1,i+2,\ldots,m)\in S_m.
\end{equation}

\begin{theorem}\label{T3.1} i).
If $\lambda$ is a tail atypical dominant integral weight, then
\begin{equation}\label{3.3}
L_\lambda=\sum_{\sigma\in \Gamma_{m-1}}\sum_{i=\flat(\sigma,\lambda)}^m
\sum_{j=\max\{0,\frac{1}{2}-\sigma(\lambda+\rho)_{i-1}\}}^{-\frac{3}{2}-\tau_i\sigma(\lambda+\rho)_{i+1}}
(-1)^{\ell(\tau_i\sigma)+j}M_{\tau_i\sigma(\lambda+\rho)-\rho+j\epsilon-j\delta_i},
\end{equation}where $\sigma(\lambda+\rho)_{0}=\infty$ and $\tau_m\sigma(\lambda+\rho)_{m+1}=-\infty$ for convenience.

ii). If $\lambda=\sum_{j=1}^m\lambda_j\delta_j+\lambda_0\epsilon$ is a $(\delta_k+\epsilon)$-atypical dominant
integral weight with $\lambda_k=\cdots\lambda_m=1$ and $\lambda_0=m-k$ (i.e. $\lambda$
is a non-tail atypical weight with $\varphi(\lambda)=\lambda^T$),
then
\begin{equation}\label{3.4}
L_\lambda=L_{\lambda^T}+\sum_{\sigma\in\Gamma_m}(-1)^{\ell(\sigma)}\mbox{ch
}M_{\sigma(\lambda+\rho)-\rho}.
\end{equation}
where $L_{\lambda^T}$ has been got by i).
\end{theorem}

We will give several lemmas first.

Denote $F_{\lambda}$ the right
side of \eqref{3.3} or \eqref{3.4}.

If $\lambda$ is a $(\delta_m-\epsilon)$-atypical integral dominant
weight of $\mathfrak{g}$, we define
\begin{eqnarray}\label{3.5}\\\nonumber
  C_{\lambda}=\frac{\prod\limits_{i=1}^m(1+x_iy)\prod\limits_{i=1}^m(1+x_iy^{-1})}
  {\prod\limits_{i=1}^m(1-x_i)\prod\limits_{1\leq i<j\leq m}(1-x_ix_j)}
    \sum_{j=0}^{\infty}\sum_{\mu}(-1)^{j+T(\mu)}S_{(j,\mu)}\sum_{s=0}^{2j}y^{j-s},
\end{eqnarray} where $\mu$ runs over all partitions
in forms
$(j_p-1,\ldots,j_1-1|j_p-1,\ldots,j_1-1)+(\lambda_{m-j_p},\ldots,\lambda_{m-j_1},-\lambda_{m-r_1},\ldots,-\lambda_{m-r_{m-1-p}})$
for any nonnegative integer $p<m$ and strictly increasing
$p$-tuples $(j_1,\ldots,j_p)\in\{1,\ldots,m-1\}^p$. Here the entries of the
strictly increasing $(m-1-p)$-tuples $(r_1,\ldots,r_{m-1-p})$ are in
the set $\{1,\ldots,m-1\}\backslash\{j_1,\ldots,j_p\}$, and
$T(\mu)=j_1+\cdots+j_p$.

If $\lambda$ is a $(\delta_m+\epsilon)$-atypical dominant integral
weight with $\lambda_m=1$ and $\lambda_0=0$, we define
\begin{eqnarray}\label{3.6}\\\nonumber
   C_{\lambda}=C_{\lambda-\delta_m}+
  \frac{\prod\limits_{i=1}^m(1+x_iy)\prod\limits_{i=1}^m(1+x_iy^{-1})}{\prod\limits_{i=1}^m(1-x_i)\prod\limits_{1\leq i<j\leq m}(1-x_ix_j)}
  \sum_{\mu}(-1)^{T(\mu)}(S_{(-1,\mu)}-S_{(0,\mu)}),
\end{eqnarray}where the definitions of $\mu$ and $T(\mu)$ are as the
same as in \eqref{3.5}.

\begin{lemma}\label{T3.2}
$C_{\lambda}=F_{\lambda}$ if one sets $x_i=e^{-\delta_i}$ and
$y=e^{\epsilon}$.
\end{lemma}
\begin{proof}
  For convenience, we denote
  $\lambda^c=\lambda_m\delta_1+\cdots+\lambda_1\delta_m$ and
  $\rho^c=-\frac{1}{2}\delta_1+\frac{1}{2}\delta_2+\cdots+(m-\frac{3}{2})\delta_m$. Then
  \begin{equation}\label{3.7}
  \Gamma_{m-1}=\{\sigma\in S_{m-1}\ltimes
\mathbb{Z}_2^{m-1}\mid
\sigma(\lambda^c+\rho^c)_2>\cdots>\sigma(\lambda^c+\rho^c)_{m}\}.
  \end{equation}
  Take
  \begin{equation}\label{3.7.1}
    \upsilon=(j_p-1,\ldots,j_1-1|j_p-1,\ldots,j_1-1)
  \end{equation}
  \begin{eqnarray}\label{3.8}\\\nonumber
\mu=\upsilon+(\lambda_{m-j_p},\ldots,\lambda_{m-j_1},-\lambda_{m-r_1},\ldots,-\lambda_{m-r_{m-1-p}}),
  \end{eqnarray} and
  \begin{equation}\label{3.9}
    \sigma_{\mu}=\omega_{\upsilon}\cdot\varepsilon_{\upsilon}
  \end{equation} for $\omega_{\upsilon}\in S_{m-1}$,
  $\varepsilon_{\upsilon}\in\mathbb{Z}_2^{m-1}$ and
  \begin{eqnarray}\label{3.10}
    \varepsilon_{\upsilon,s}=\left\{\begin{array}{ll}
      -1& \mbox{if}\ \  s\in\{m-j_1,\ldots,m-j_p\}\\
      1& \mbox{if}\ \  s\in\{m-r_1,\ldots,m-r_{m-1-p}\}
    \end{array}\right..
  \end{eqnarray} One easily shows that
  \begin{equation}\label{3.11}
    (-1)^{j+T(\mu)}S_{(j,\mu)}=(-1)^{j+l(\sigma_{\mu})}S_{\rho^c-\sigma_{\mu}(\lambda^c+\rho^c)+j\delta_1}.
  \end{equation} Since there is a 1-1 correspondence between the
  elements in $\Gamma_{m-1}$ and the
  strict increasing $p-$tuples $(j_1,\ldots,j_p)\in\{1,\ldots,m-1\}^p$ (see Remak \ref{r3.1}), where $p$ runs over the set
  $\{1,\ldots,m-1\}$, we get that
  \begin{eqnarray}\label{3.12}
    \sum_{\mu}(-1)^{j+T(\mu)}S_{(j,\mu)}=
    \sum_{\sigma\in\Gamma_{m-1}}(-1)^{j+l(\sigma)}S_{\rho^c-\sigma(\lambda^c+\rho^c)+j\delta_1}.
  \end{eqnarray} For an $m-$tuple $(\nu_1,\ldots,\nu_m)$, one has
  \begin{equation}\label{3.13}
    S_{(\nu_1,\ldots,\nu_m)}(x_1,\ldots,x_m)=
    S_{(-\nu_m,\ldots,-\nu_1)}(x_1^{-1},\ldots,x_m^{-1}).
  \end{equation} Thus
  \begin{equation}\label{3.14}
    S_{\rho^c-\sigma(\lambda^c+\rho^c)+j\delta_1}(x_1,\ldots,x_m)=
     S_{\sigma(\lambda+\rho)-\rho-j\delta_m}(x_1^{-1},\ldots,x_m^{-m}).
  \end{equation} Furthermore, it is well known that
  $S_{\eta}=(-1)^{l(w)}S_{\theta}$ if
  $\eta+\rho_{\mathfrak{gl}(m)}=w(\theta+\rho_{\mathfrak{gl}(m)})$, where $w\in S_m$. One
  checks that $S_{\sigma(\lambda+\rho)-\rho-j\delta_m}=0$ if
  $j\notin
  \{\max\{0,\frac{1}{2}-\sigma(\lambda+\rho)_{i-1}\},\ldots,-\frac{1}{2}-\tau_i\sigma(\lambda+\rho)_{i+1}\}$
  for $i\in\{\flat(\sigma,\lambda),\ldots,m\}$ in \eqref{3.3}. These
  complete the proof.
\end{proof}

\begin{lemma}\label{1} The following identity holds:
  \begin{equation}\label{2}
    \frac{\prod\limits_{i=1}^m(1-x_i)\prod\limits_{1\leq i<j\leq m}(1-x_ix_j)}{\prod\limits_{i=1}^m(1+x_iy)\prod\limits_{i=1}^m(1+x_iy^{-1})}=
    \sum_{j=0}^{\infty}\sum_{\mu}(-1)^{j+{\frac{|\mu|+p(\mu)}{2}}}S_{(j,\mu)}\sum_{s=0}^{2j}y^{j-s},
  \end{equation} where $\mu$ runs over all self-conjugate partitions $\mu=(\alpha_1,\ldots,\alpha_p|\alpha_1,\ldots,\alpha_p)$ with $\alpha_1\leq
  m-2$. Here $p(\mu)=p$ and $S_{(j,\mu)}=S_{(j,\mu)}(x_1,\ldots,x_m)$.
\end{lemma}
\begin{proof}
  One shows that
  \begin{equation}\label{3}
    \prod\limits_{i=1}^m(1+x_iy)\prod\limits_{i=1}^m(1+x_iy^{-1})=\sum_{k=0}^m(\sum_{j=0}^ke_je_{m-k+j})y^{m-k}
    +\sum_{k=1}^m(\sum_{j=0}^{m-k}e_{k+j}e_j)y^{-k},
  \end{equation} where $e_j=e_j(x_1,\ldots,x_m)$.

  We denote
  \begin{equation}\label{4}
    A_j=\sum_{\stackrel{\mu=(\alpha_1,\ldots,\alpha_p|\alpha_1,\ldots,\alpha_p)}{0\leq\alpha_1\leq m-2}}
    (-1)^{\frac{|\mu|+p(\mu)}{2}}S_{(j,\mu)}
  \end{equation} On the other hand, we can write
  $A_j=\sum_{s=0}^{m-1}a_sh_{j+s}$ by the definition of $S_{\lambda}$, where $a_s$ are certain
  polynomials in variables $x_1, \ldots, x_m$, and $h_j=h_j(x_1,\ldots,
  x_m)$. We write
  \begin{equation}\label{5}
    B_k=\left\{
    \begin{array}{ll}
      \sum_{j=k}^{\infty}(-1)^jA_j\ \ & \ \ if \ \ k\geq0,\\
      \sum_{j=-k}^{\infty}(-1)^jA_j\ \ & \ \ if \ \ k<0.
    \end{array}\right.
  \end{equation} Thus $B_k=B_{-k}.$ Moreover
  \begin{eqnarray}\label{6}
    \sum_{j=0}^{\infty}\sum_{\mu}(-1)^j(-1)^{\frac{|\mu|+p(\mu)}{2}}S_{(j,\mu)}\sum_{s=0}^{2j}y^{j-s}
    =\sum_{k=-\infty}^{\infty}B_ky^k.
  \end{eqnarray}
Hence
\begin{eqnarray}\label{7}
  &&\prod\limits_{i=1}^m(1+x_iy)\prod\limits_{i=1}^m(1+x_iy^{-1})(\sum_{k=-\infty}^{\infty}B_ky^k)\nonumber\\
  &=&\sum_{l=-\infty}^{\infty}(\sum_{k=0}^m\sum_{j=0}^ke_je_{m-k+j}B_{l-m+k}+\sum_{k=1}^m\sum_{j=0}^{m-k}e_{k+j}e_jB_{l+k})y^l.
\end{eqnarray}

Now we consider the coefficients of $y^l$. Since the coefficient of
$y^l$ is the same as of $y^{-l}$, we only need to check the case
$l\geq 0$. In fact, we want to show these coefficients are 0 if
$l>0$, and is $\prod_{i=1}^m(1-x_i)\prod_{1\leq i<j\leq
m}^m(1-x_ix_j)$ if $l=0$.

Note that
\begin{eqnarray}\label{8}\\\nonumber
  \sum_{k=0}^m\sum_{j=0}^ke_je_{m-k+j}B_{l-m+k}+\sum_{k=1}^m\sum_{j=0}^{m-k}e_{k+j}e_jB_{l+k}=
  \sum_{j=0}^me_j\sum_{k=0}^me_{m-k}B_{l-(m-j)+k},
\end{eqnarray}
\begin{equation}\label{9}
B_p=\sum_{j=p}^{\infty}(-1)^jA_j=\sum_{j=p}^{\infty}\sum_{s=0}^{m-1}(-1)^ja_sh_{j+s}
\end{equation} for $p\geq0$, and
\begin{equation}\label{10}
  h_k-e_1h_{k-1}+e_2h_{k-2}+\cdots+(-1)^me_mh_{h-m}=0
\end{equation} for $k\neq0$. Thus the coefficient of $y^l$ is 0 if $l\geq
m$.

Let $l=m-1$, It is sufficiently to show that
\begin{equation}\label{11}
  e_mB_{-1}+e_{m-1}B_0+\cdots+e_1B_{m-2}+B_{m-1}=0.
\end{equation}
Since
\begin{equation}\label{12}
   e_mB_{0}+e_{m-1}B_1+\cdots+e_1B_{m-1}+B_{m}=0
\end{equation} by\eqref{8}-\eqref{10}, then \eqref{11} is equal to
\begin{equation}\label{13}
  A_0=A_{-1}.
\end{equation} By induction, one shows that the equality\eqref{2} is equal
to
\begin{equation}\label{14}
  \left\{\begin{array}{ll}
    A_j=A_{-(j+1)} & \mbox{for}\ j=0,1,\ldots,m-2,\\
    (-1)^mA_{m-1}+a_0=\prod\limits_{i=1}^m(1-x_i)\prod\limits_{1\leq i<j\leq
    m}(1-x_ix_j).
  \end{array}\right.
\end{equation}Note that\cite{M}
\begin{equation}\label{15}
  \prod_{i=1}^m(1-x_i)\prod_{1\leq i<j\leq
    m}(1-x_ix_j)=\sum_{\mu}(-1)^{\frac{|\mu|+p(\mu)}{2}}S_{\mu},
\end{equation} where
$\mu=(\alpha_1,\ldots,\alpha_p|\alpha_1,\ldots,\alpha_p)$ and
$\alpha_1\leq m-1.$ But
\begin{equation}\label{16}
  a_0=\sum_{\mu}(-1)^{\frac{|\mu|+p(\mu)}{2}}S_{\mu}
\end{equation} where
$\mu=(\alpha_1,\ldots,\alpha_p|\alpha_1,\ldots,\alpha_p)$ and
$\alpha_1\leq m-2$. Then
\begin{eqnarray}\label{17}
&&\prod_{i=1}^m(1-x_i)\prod_{1\leq i<j\leq
    m}(1-x_ix_j)-a_0\\\nonumber
&=&e_m\sum_{\mu=(\alpha_1,\ldots,\alpha_p|\alpha_1,\ldots,\alpha_p)}
(-1)^{\frac{|\mu|+p(\mu)+2m}{2}}S_{(m-1,\mu)}\\\nonumber
&=&(-1)^me_mA_{m-1}.
\end{eqnarray} Now the statement $\eqref{14}_2$ is hold.

Observe that
\begin{eqnarray}\label{18}
  &&(\alpha_1,\ldots,\alpha_p|\alpha_1,\ldots,\alpha_p)\\\nonumber
&=&(\alpha_1+1,\ldots,\alpha_p+p,\overbrace{p,\ldots,p}^{\alpha_p}
,\overbrace{p-1,\ldots,p-1}^{\alpha_{p-1}-\alpha_p-1},\ldots,\overbrace{1,\ldots,1}^{\alpha_1-\alpha_2-1},\overbrace{0,\ldots,0}^{m-\alpha_1-2})
\end{eqnarray}
Thus the first column of the corresponding matrix for
$S_{(j,(\alpha_1,\ldots,\alpha_p|\alpha_1,\ldots,\alpha_p))}$ is
\begin{eqnarray}\label{19}\\\nonumber
  (h_j,h_{\alpha_1},\ldots,h_{\alpha_p},\overbrace{h_{-1},\ldots,h_{-\alpha_p}}^{\alpha_p},
  \overbrace{h_{-\alpha_p-2},\ldots,h_{-\alpha_{p-1}},}^{\alpha_{p-1}-\alpha_p-1}\ldots,\\\nonumber
  \overbrace{h_{-\alpha_2-2},\ldots,h_{-\alpha_1},}^{\alpha_1-\alpha_2-1}
  \overbrace{h_{-\alpha_1-2},\ldots,h_{-m+1}}^{m-\alpha_1-2}).
\end{eqnarray}
It is clearly that
$S_{(j,(\alpha_1,\ldots,\alpha_p|\alpha_1,\ldots,\alpha_p))}\neq0$
if and only if $\alpha_s\in\{0,\ldots,m-2\}\backslash\{j\}$ for $0\leq j\leq
m-2$, where $s\in\{1,\ldots,p\}$. In this case, there exists $s$
s.t. $\alpha_{s+1}<j$ and $\alpha_s>j$ for $s\in\{1,\ldots,p\}$.
Hence
\begin{equation}\label{20}
\alpha_{s+1}+2\leq j+1,\ \ \ \ \alpha_s\geq j+1.
\end{equation} Since
$(|(\alpha_1,\ldots,\alpha_p\mid\alpha_1,\ldots,\alpha_p)|+p)/2=p+\alpha_1+\cdots+\alpha_p$,
thus
\begin{eqnarray}\label{21}
&&(-1)^{p+\alpha_1+\cdots+\alpha_p}S_{(j,(\alpha_1,\ldots,\alpha_p|\alpha_1,\ldots,\alpha_p))}\\\nonumber&=&
(-1)^{p+\alpha_1+\cdots+\alpha_p+j+1}
S_{(-(j+1),(\alpha_1,\ldots,\alpha_s,j,\alpha_{s+1},\ldots,\alpha_p|\alpha_1,\ldots,\alpha_s,j,\alpha_{s+1},\ldots,\alpha_p))}.
\end{eqnarray}
We see that
$S_{(-(j+1),(\alpha_1,\ldots,\alpha_p|\alpha_1,\ldots,\alpha_p))}\neq0$
if and only if there exist $s\in\{\alpha_1,\ldots,\alpha_p\}$ s.t. $\alpha_s=j$
by \eqref{20}, which implies that $A_j=A_{-(j+1)}$ for
$j=0,1,\ldots,m-2$. These complete the proof.
\end{proof}

In fact, Lemma \ref{1} gives the character of the trivial
module $L_0$ in form of an infinite sum of characters of generalized Verma modules.

\begin{lemma}\label{27} For $\lambda\in\mathcal{P}$ and
$\lambda_m=\lambda_0=0$, we have
  \begin{equation}\label{28}
  C_{\lambda}\cdot C_{\delta_1}=\sum_{\mu\in\mathcal{P}_{\lambda}}C_{\mu}.
  \end{equation}
\end{lemma}
\begin{proof}
  Denote
  $T=S_{(1)}+S_{(0,\ldots,0,-1)}+y+1+y^{-1}$. We want to show that
\begin{equation}\label{29}
  C_{\lambda}\cdot T=\sum_{\mu\in\mathcal{P}_{\lambda}}C_{\mu}.
  \end{equation} Since $C_0=1$ by Lemma \ref{1}, we get that
  $T=C_{\delta_1}$.

  For convenience, we will write
  \begin{equation}\label{30}
    A=\frac{\prod\limits_{i=1}^m(1+x_iy)\prod\limits_{i=1}^m(1+x_iy^{-1})}{\prod\limits_{i=1}^m(1-x_i)\prod\limits_{1\leq i<j\leq
    m}(1-x_ix_j)}.
  \end{equation} Thus
  \begin{eqnarray}\label{31}\\\nonumber
    &&C_{\lambda}\cdot
    T\\\nonumber&=&A\sum_{j=0}^{\infty}\sum_{\mu}\sum_{l=2}^m(-1)^{j+T(\mu)}S_{(j,\mu)\pm\delta_l}\sum_{s=0}^{2j}y^{j-s}
    +A\sum_{j=0}^{\infty}\sum_{\mu}(-1)^{j+T(\mu)}S_{(j,\mu)}\sum_{s=0}^{2j}y^{j-s}\\\nonumber
    &+&A\sum_{\mu}(-1)^{T(\mu)}S_{(-1,\mu)}-A\sum_{\mu}(-1)^{T(\mu)}S_{(0,\mu)}.
  \end{eqnarray}

One can easily show that $\sum_{j=1}^mS_{\lambda\pm\delta_j}=0$ if
$S_{\lambda}=0$. Moreover,
$\sum_{j=1}^mS_{\lambda\pm\delta_j}=\sum_{j=1}^mS_{\lambda\pm\sigma(\delta_j)}$
for any $\sigma\in S_m\ltimes\mathbb{Z}_2^m$.

Using \eqref{3.11}, we get that
\begin{eqnarray}\label{3.15}
&&C_{\lambda}\cdot
    T\\\nonumber&=&A\sum_{j=0}^{\infty}\sum_{\sigma\in(S_{m-1}\ltimes\mathbb{Z}_2^{m-1})^{\lambda}}
    \sum_{l=2}^m(-1)^{j+l(\sigma)}S_{\rho^c-\sigma(\lambda^c+\rho^c\pm\delta_l)+j\delta_1}\sum_{s=0}^{2j}y^{j-s}\\\nonumber
    &+&A\sum_{j=0}^{\infty}\sum_{\sigma\in(S_{m-1}\ltimes\mathbb{Z}_2^{m-1})^{\lambda}}
    (-1)^{j+l(\sigma)}S_{\rho^c-\sigma(\lambda^c+\rho^c)+j\delta_1}\sum_{s=0}^{2j}y^{j-s}\\\nonumber
    &+&A\sum_{\sigma\in(S_{m-1}\ltimes\mathbb{Z}_2^{m-1})^{\lambda}}(-1)^{l(\sigma)}
    (S_{\rho^c-\sigma(\lambda^c+\rho^c)-\delta_1}-S_{\rho^c-\sigma(\lambda^c+\rho^c)}).
\end{eqnarray}

 We say that
$\lambda^c\in\mathcal {P}_{\mathfrak{gl}(m)}^c$ if
$\lambda\in{\mathcal {P}}_{\mathfrak{gl}(m)}$ (note that
$\lambda_0=0$).

If $\lambda^c+\delta_{l+1}\notin\mathcal {P}^c_{\mathfrak{gl}(m)}$,
i.e. $\lambda+\delta_{m-l}\notin\mathcal {P}_{\mathfrak{gl}(m)}$ for
$l=1,\ldots,m-2$, then $\lambda_{m-l}=\lambda_{m-(l+1)}$.

Let $l\in\{j_1,\ldots,j_p\}$, for example, $l=j_{p-t}$. Now if
$l+1\in\{j_1,\ldots,j_p\}$, then $l+1=j_{p-t+1}$. Thus
$j_{p-t}+t+\lambda_{m-j_{p-t}}=j_{p-t+1}+t-1+\lambda_{m-j_{p-t+1}}$.
It implies that
$S_{\sigma(\lambda+\rho+\delta_{m-l})-\rho-j\delta_m}=0$. If
$l+1\in\{r_1,\ldots,r_{m-1-p}\}$, we set $j'_k=j_k$ if $k\neq p-t$,
and $j'_{p-t}=l+1$. Thus there exists $s\in\{1,\ldots,m-1-p\}$ s.t.
$r_s=l+1$. We set $r'_q=r_q$ if $q\neq s$, and $r'_s=l$. We denote
\begin{equation}\label{34}
w=(m-p-s,m-p+t)(0,1,\ldots,1,\overset{m-p-s}{-1},1,\ldots,1,\overset{m-p+t}{-1},1,\ldots,1).
\end{equation} Then
$\sigma'=w\sigma$(note that
$\sigma'\in\Gamma_{m-1}$), and
\begin{equation}\label{3.16}
  \sigma(\lambda+\rho+\delta_{m-l})-\rho-j\delta_m=\sigma'(\lambda+\rho+\delta_{m-l})-\rho-j\delta_m.
\end{equation}
But $(-1)^{l(\sigma')}=-(-1)^{l(\sigma)}$, thus we know that
\begin{equation}\label{3.17}
  (-1)^{j+l(\sigma)}S_{\rho^c-\sigma(\lambda^c+\rho^c+\delta_{l+1})+j\delta_1}
+(-1)^{j+l(\sigma')}S_{\rho^c-\sigma'(\lambda^c+\rho^c+\delta_{l+1})+j\delta_1}=0
\end{equation}
in \eqref{3.15}. The case $l\in\{r_1,\ldots,r_{m-1-p}\}$ is similar.

If $\lambda^c-\delta_{l+2}\notin\mathcal {P}^c_{\mathfrak{gl}(m)}$,
i.e. $\lambda-\delta_{m-l-1}\notin\mathcal {P}_{\mathfrak{gl}(m)}$
for $l=1,\ldots,m-2,$ the discussion is similar to the above.

If $\lambda_{m-1}=0$, then $\lambda-\delta_{m-1}\notin\mathcal
{P}_{\mathfrak{gl}(m)}$, i.e. $\lambda^c-\delta_2\notin\mathcal
{P}^c_{\mathfrak{gl}(m)}$. Note that
$\sigma(\lambda+\rho-\delta_{m-1})=\sigma(1,,\ldots,1,-1,1)(\lambda+\rho)$,
and if $\sigma\in\Gamma_{m-1}$, then $\sigma'\in\Gamma_{m-1}$ by
$\lambda_{m-1}=0$ and $\rho_{m-1}=\frac{1}{2}$. These imply that
\begin{eqnarray}\label{3.18}\\\nonumber
\sum\limits_{\sigma\in\Gamma_{m-1}}(-1)^{j+l(\sigma)}S_{\rho^c-\sigma(\lambda^c+\rho^c-\delta_2)+j\delta_1}
+\sum\limits_{\sigma\in\Gamma_{m-1}}(-1)^{j+l(\sigma)}S_{\rho-\sigma(\lambda^c+\rho^c)+j\delta_1}=0.
\end{eqnarray}

One shows
that$\sum_{\mu}(-1)^{j_1+\cdots+j_p}S_{(-1,\mu)}-\sum_{\mu}(-1)^{j_1+\cdots+j_p}S_{(0,\mu)}=0$
if $\lambda_{m-1}=0$, via $S_{(0,\mu)}\neq 0$ iff $j_1\neq1$ and
$S_{(-1,\mu)}\neq 0$ iff $j_1=1$. Thus \eqref{28} hold.
\end{proof}

\vspace{0.5 cm}

\noindent\textbf{Proof of Theorem \ref{T3.1}:} We will prove the
case $\lambda_0=0$ in this theorem by induction on the height of
weight $\lambda$. Firstly, if $ht \lambda=0$ (i.e. $\lambda=0$), one
easily shows that the statement holds by Lemma \ref{T3.2} and Lemma
\ref{1}. Moreover, the statement holds for $\lambda=\delta_1$ by
Lemma \ref{T3.2}. Now by induction,
\begin{eqnarray}\label{3.19}
  L_{\lambda}\otimes L_{\delta_1}=\sum_{\mu\in\mathcal{P}_{\lambda}}C_{\mu}\quad\mbox{if
  $\lambda_m=0$}.
\end{eqnarray} In this case, $\lambda$ is a tail
atypical integral dominant weight. Since the weights in $\mathcal
{P}_{\lambda}$ are in different blocks and
$\tau_i\sigma(\lambda+\rho)-\rho-j\delta_i<\lambda$, then we have
$L_{\mu}=C_{\mu}$ if $\mu\in\mathcal{P}_{\lambda}$ and if $\mu$ is
also a tail atypical integral dominant weight by Remark 2.5. Now if
$\lambda_{m-1}>0$, we consider $L_{\lambda+\delta_m}$. By
\eqref{3.19} and above discussion, we can assume that
\begin{eqnarray}\label{3.20}
  L_{\lambda+\delta_m}+xL_{\lambda}=L_{\lambda}+\sum_{\sigma\in\Gamma_m}(-1)^{\ell(\sigma)}
M_{\sigma(\lambda+\delta_m+\rho)-\rho},
\end{eqnarray}where $x$ is a nonnegative integer. One compares the coefficients of $e^{\lambda}$ of the
two sides of \eqref{3.20} then gets that $x<2$. But if $x=1$, one
has
$L_{\lambda+\delta_m}=\sum\limits_{\sigma\in\Gamma_m}(-1)^{\ell(\sigma)}
M_{\sigma(\lambda+\delta_m+\rho)-\rho}$. It contradicts to Lemma
\ref{T2.8.2}. These prove the statement of the case $\lambda_0=0$.

Now we show $ii)$ of the theorem with the case $\lambda_0>0$ (i.e. $k<m$ in \eqref{3.4}). Denote $\lambda=\sum_{j=1}^{k-1}\lambda_j\delta_j+\sum_{j=k}^m\delta_j+(m-k)\epsilon$ and  $\lambda'=\sum_{j=1}^{k-1}\lambda_j\delta_j+\sum_{j=k}^m\delta_j+(m-k+1)\epsilon$  . Then we claim that
\begin{equation}\label{3.21}
  F_{\lambda'}\cdot L_{\delta_1}=\sum_{\mu\in\mathcal{P}_{\lambda'}\backslash\{\lambda'\}}F_{\mu}.
\end{equation} This claim can be show by induction on $k$. Note that
\begin{eqnarray}\label{3.22}
  \mbox{Left}=\sum_{\nu\in\mathcal{P}_{\lambda^T}}
  L_{\nu}+\sum_{\sigma\in\Gamma_m}\sum_{\xi\in\mathcal{P}^{(0)}_{\sigma(\lambda'+\rho)-\rho}}(-1)^{l(\sigma)}M_{\xi},
\end{eqnarray}
\begin{eqnarray}\label{3.23}
  \mbox{Right}=\sum_{\nu\in\mathcal{P}_{\lambda^T}}
  L_{\nu}+\sum_{\sigma\in\Gamma_m}\sum_{\eta\in\mathcal{P}_{\sigma(\lambda'+\rho)-\rho}\backslash\{\sigma(\lambda'+\rho)-\rho\}}(-1)^{l(\sigma)}M_{\eta}.
\end{eqnarray} The elements in $\mathcal{P}^{(0)}_{\sigma(\lambda'+\rho)-\rho}$ but not in $\mathcal{P}_{\sigma(\lambda'+\rho)-\rho}\backslash\{\sigma(\lambda'+\rho)-\rho\}$ are
$\nu=\sigma(\sum_{j=1}^{k-1}\lambda_j\delta_j+\sum_{j=k}^{m-1}\delta_j+(m-k+1)\epsilon+\rho)-\rho$ and $\zeta=\sigma(\sum_{j=1}^{k-1}\lambda_j\delta_j+\sum_{j=k}^m\delta_j+(m-k+1)\epsilon+\rho)-\rho$. Since $\sigma\cdot(1,\ldots,1,-1)(\nu)=\zeta$ and
$\sigma\cdot(1,\ldots,1,-1)\in\Gamma_m$ if $\sigma\in\Gamma_m$. This implies that $\sum_{\sigma\in\Gamma_m}(-1)^{l(\sigma)}M_{\nu}+\sum_{\sigma\in\Gamma_m}(-1)^{l(\sigma)}M_{\zeta}=0$. Thus the claim is hold.

Now we show $L_{\lambda}=F_{\lambda}$ by induction on $m-k$. By \eqref{3.21}, one has
\begin{equation}\label{3.24}
L_{\lambda}=(1-x)L_{\lambda^T}+\sum_{\sigma\in\Gamma_m}(-1)^{l(\sigma)}M_{\sigma(\lambda+\rho)-\rho},
\end{equation} where $x\geq0$.

If $1<k<m$, the coefficient of $M_{\lambda^T+\delta_k}$ in
$L_{\lambda'}$ or $L_{\lambda'^T}$ is nonnegative and in
$\sum_{\sigma\in\Gamma_m}(-1)^{l(\sigma)}M_{\sigma(\lambda+\rho)-\rho}\otimes
L_{\delta_1}$ is 0. One can tensor $L_{\delta_1}$ on both sides of
\eqref{3.24} to show that $x\leq1$. Lemma \ref{T2.8.2} says that
$x=0$.

For $k=1$, note that $[L_{\sum_{j=1}^m\delta_j+(m-1)\epsilon}\otimes L_{\delta_1}:L_{\sum_{j=1}^m\delta_j+(m-2)\epsilon}]=1$. Denote $y=[L_{\sum_{j=1}^m\delta_j+(m-1)\epsilon}\otimes L_{\delta_1}:L_{\delta_1}]\geq0$. On the other hand,
\begin{equation}\label{3.25}
 L_{\sum_{j=1}^m\delta_j+(m-1)\epsilon}\otimes L_{\delta_1}=(1-x)L_0\otimes L_{\delta_1}+\sum_{\sigma\in\Gamma_m}(-1)^{l(\sigma)}M_{\sigma(\sum_{j=1}^m\delta_j+(m-1)\epsilon+\rho)-\rho}\otimes L_{\delta_1}.
\end{equation}  Note that by induction,
\begin{equation}\label{3.26}
  L_{\sum_{j=1}^m\delta_j+(m-2)\epsilon}=L_{\delta_1}+\sum_{\sigma\in\Gamma_m}(-1)^{l(\sigma)}M_{\sigma(\sum_{j=1}^m\delta_j+(m-2)\epsilon+\rho)-\rho}.
 \end{equation} One substitutes \eqref{3.26} into \eqref{3.25} to show that $y=-x=0$. These complete the proof.
\ \ \ \ \ \ \ \ \ \ \ \ \ \ \ \ \ \ \ \ \ \ \ \ \ \ \ \ \ \ \ \ \ \ \ \ \ \ \ \ \ \ \ \ \ \ \ \ \ \ \ \ \ \ \ \ \ \ \ \ \ \ \ \ \ \ \ \ \ \ \ \ \ \ \ \ \ \ \ \ \ \ \ \ \ \ \ \ \ \ \ \ \ \ \ \ \ \ \ $\Box$

\section{Structure of tensor module $L_\lambda\otimes L_{\delta_1}$}
\begin{lemma}
For any $\mu\in\mathcal {P}_{\lambda^+}$, it should be that
\begin{equation}
[L_\lambda\otimes L_{\delta_1}: L_\mu]=1.
\end{equation} Particularly, if $\mu$ is typical, then $L_\mu$ is a
direct summand in $L_\lambda\otimes L_{\delta_1}$.
\end{lemma}
\begin{proof}
If $\lambda$ is a tail atypical weight, the statement has been shown
in the proof of Theorem \ref{T3.1}. Now assume $\lambda$ is either
typical or $(\delta_t+\epsilon)$-atypical ($1\leq t\leq m$). Thus
$\chi_\lambda\neq\chi_{\lambda-\alpha}$ for any
$\alpha\in\{\delta_j-\delta_k,\delta_l-\epsilon\mid 1\leq j<k\leq m,
1\leq l\leq m\}$. Hence for any $\nu\prec\lambda$ (i.e.
$\chi_\nu=\chi_\lambda$ and $\nu<\lambda$), it should be that
\begin{equation}
\{\nu\pm\delta,\nu\pm\epsilon_j\mid 1\leq j\leq m\}\cap\mathcal
{P}_{\lambda^+}=\emptyset.
\end{equation} Therefore if we multiply $L_{\delta_1}$ on the both
sides of (2.31) and calculate the right side by (3.6), then we can
obtain that the coefficient of $M_\mu$ ($\mu\in\mathcal
{P}_{\lambda^+}$) is exactly $1$. Thanks to Remark 2.6, we get
\begin{equation}
[L_\lambda\otimes L_{\delta_1}: L_\mu]=1\quad\mbox{for any
$\mu\in\mathcal {P}_{\lambda^+}$.}
\end{equation}
\end{proof}

\begin{lemma}
For any $\lambda\in\mathcal {P}$, if $L_\mu$ is an irreducible
submodule or quotient module of $L_\lambda\otimes L_{\delta_1}$,
then it must be that $\mu\in\mathcal {P}_\lambda$.
\end{lemma}
\begin{proof}
If $\lambda$ is a tail atypical weight, then one can check that the
statement hold by Theorem \ref{T3.1}.

Now suppose that $\lambda$ is either typical or
$(\delta_t+\epsilon)$-atypical ($1\leq t\leq m$). Thus for any
$\nu\prec\lambda$, there is no weight $\mu\in\mathcal {P}_\nu^{0}$
such that $\mbox{ht }\mu> \mbox{ht }\lambda$. Hence if we multiply
$L_{\delta_1}$ on the right side of (2.31), then the coefficient of
$M_\mu$, where $\mu$ satisfies $\mbox{ht }\mu>\mbox{ht }\lambda$, is
nonzero if and only if $\mu\in\mathcal {P}_{\lambda^+}$. So by
Remark 2.6, we have
\begin{equation}
[L_\lambda\otimes L_{\delta_1} : L_\mu]=0 \quad \mbox{for any
$\mu\in\mathcal {P}\setminus \mathcal {P}_\lambda$ with $\mbox{ht
}\mu>\mbox{ht }\lambda$.}
\end{equation}

Take any irreducible submodule or quotient module $L_\mu$ of
$L_\lambda\otimes L_{\delta_1}$.

Suppose $\mbox{ht }\mu>\mbox{ht }\lambda$. It must be that
$\mu\in\mathcal {P}_\lambda$ because of (4.4).

Suppose $\mbox{ht }\mu<\mbox{ht }\lambda$. Since
\begin{equation}
\mbox{Hom}_\mathfrak{g}(L_\lambda\otimes L_{\delta_1},
L_\mu)\simeq\mbox{Hom}_\mathfrak{g}(L_\lambda, L_\mu\otimes
L_{\delta_1})
\end{equation} and
\begin{equation}
\mbox{Hom}_\mathfrak{g}(L_\mu,L_\lambda\otimes L_{\delta_1}
)\simeq\mbox{Hom}_\mathfrak{g}(L_\mu\otimes L_{\delta_1},
L_\lambda),
\end{equation}
it should be that $L_\lambda$ is an irreducible submodule or
quotient module of $L_\mu\otimes L_{\delta_1}$. Thus also by (4.4),
we have $\lambda\in \mathcal{P}_\mu$, which implies that $\mu\in
\mathcal{P}_\lambda$.

Suppose $\mbox{ht }\mu=\mbox{ht }\lambda$.  There can not be a
weight $\nu\succ\mu$ (note that $\mbox{ht } \nu-2\geq \mbox{ht
}\mu=\mbox{ht }\lambda$) such that $[L_\lambda\otimes L_{\delta_1}:
L_\nu]\neq0$. Thus there should be a weight vector with highest
weight $\mu$. So if we multiply $L_{\delta_1}$ on the both sides of
(2.31), then on the right side the coefficient of $M_\mu$ is
nonzero. But it is clear by (3.6) that, for any $\nu\prec\lambda$
(note that $\mbox{ht } \nu\leq \mbox{ht }\lambda-2$), the
coefficient of $M_\mu$ in $M_\nu\otimes L_{\delta_1}$ is zero. So
$M_\mu$ appears in $M_\lambda\otimes L_{\delta_1}$. It implies that
$\mu\in\mathcal {P}_\lambda$.
\end{proof}

\begin{lemma}
Suppose $\lambda\in\mathcal {P}$ is an atypical weight. For any
$\mu,\nu\in\mathcal {P}_\lambda$ with $\mu\neq\nu$, it must be that
$\chi_\mu\neq\chi_\nu$.
\end{lemma}
\begin{proof} Take any $\lambda\in\mathcal {P}$. If $\lambda_{m-1}=0$
(which also implies $\lambda_m=\lambda_0=0$ by Theorem 2.1), one can check easily that the
statement holds.

Now suppose that $\lambda_{m-1}\neq0$, then
$|\widetilde{\lambda}_1|,\ldots,|\widetilde{\lambda}_{m}|$ are
pairwise distinguished.

If there exist $\mu\neq\nu\in\mathcal {P}_\lambda$ such that
$\chi_\mu=\chi_\nu$. By Lemma 2.3, we only need to consider the case
that $\mu$ and $\nu$ are both atypical. It is obvious by Lemma 2.3
that $\lambda-\mu, \lambda-\nu\in \{\pm\epsilon,0\}$. Hence
$|\widetilde{\mu}_0|=|\widetilde{\nu}_0|$. If either $\lambda-\mu$
or $\lambda-\nu$ is $0$, then $\widetilde{\lambda_0}>1$. Thus
$|\widetilde{(\lambda\pm\epsilon)}_0|\neq|\widetilde{\lambda}_0|$,
which is a contradiction to
$|\widetilde{\mu}_0|=|\widetilde{\nu}_0|$. So $\mu,
\nu=\lambda\pm\epsilon$, which is also a contradiction to
$|\widetilde{\mu}_0|=|\widetilde{\nu}_0|$ since
$|\widetilde{\mu}_0-\widetilde{\nu}_0|=2$ and $\widetilde{\nu}_0,
\widetilde{\mu}_0\in\mathbb{Z}+\frac{1}{2}$.
\end{proof}

\begin{corollary}
For any atypical weight $\lambda\in\mathcal {P}$, if $\mu\in\mathcal
{P}_\lambda$ is also an atypical weight, then $L_\mu$ is a direct
summand in $L_\lambda\otimes L_\delta$ and
\begin{equation}
[L_\lambda\otimes L_\delta:L_\mu]=1.
\end{equation}
\end{corollary}
\begin{proof}
Just combine Lemmas 4.1, 4.2 and 4.3.
\end{proof}

\begin{remark}
Lemma 4.3 and Corollary 4.4 indicate that for any two atypical
weights $\lambda,\mu\in\mathcal {P}$, if there exist atypical
weights $\mu^{(0)}=\lambda,\mu^{(1)},\cdots,\mu^{(t)}=\mu\in\mathcal
{P}$ such that $\mu^{(i+1)}\in\mathcal {P}_{\mu^{(i)}}$, then one
can use (3.7) iteratively to get $\mbox{ch }L_\mu$ from $\mbox{ch
}L_\lambda$ by a straightforward calculation:
\begin{equation}
L_{\mu^{(i)}}\xrightarrow{\otimes L_\delta} L_{\mu^{(i+1)}}.
\end{equation} Such an algorithm will be used in the next
section.
\end{remark}

\begin{remark}
All of the results stated in this section are also valid for
$\mathfrak{osp}(n|2)$ (cf. \cite{L}).
\end{remark}

\section{Character formulae for non-tail atypical weights}
\subsection{When $\varphi(\lambda)=\lambda^T$}
\

We have dealt with this case in
Theorem \ref{T3.1} ii).

\subsection{When $\varphi(\lambda)\neq\lambda^T$}
\subsubsection{\bf In case that $\lambda$ is $(\delta_m+\epsilon)$-atypical}
In this case, $\lambda_m=\lambda_0+1>1$ and $\lambda^T={\lambda-\lambda_0(\delta_m+\epsilon)-\delta_m}$. Assume that
$\lambda_0=s>0$.

Let $\mu=\lambda-\epsilon$. It is obvious that $\mu\in\mathcal {P}$
is typical. So by Lemma 2.5, we have
\begin{equation}
L_\mu=\sum_{\sigma\in\Gamma_m}(-1)^{\ell(\sigma)} M_{\sigma(\mu+\rho)-\rho}.
\end{equation}

Suppose
\begin{equation}
[L_\mu\otimes L_{\delta_1}: L_{\lambda-i(\delta_m+\epsilon)}]=x_i,
(0\leq i\leq s);\quad [L_\mu\otimes L_{\delta_1}:
L_{\lambda-s(\delta_m+\epsilon)-\delta_m}]=y.
\end{equation} Note that $x_0=1$ by Lemma 4.1.

Multiply $L_{\delta_1}$ on the both sides of the above formula. By
(3.7), we can calculate that the right side becomes to
\begin{equation}
\sum_{\nu\in\mathcal {P}_\mu}\sum_{\sigma\in\Gamma_m}(-1)^{\ell(\sigma)} M_{\sigma(\nu+\rho)-\rho}.
\end{equation}

Select in (5.3) the modules $M_\nu$ with $\chi_\nu=\chi_\lambda$.
Then we have
\begin{equation}
L_\lambda+\sum_{i=1}^{s}x_iL_{\lambda-i(\delta_m+\epsilon)}+yL_{\lambda-s(\delta_m+\epsilon)-\delta_m}
=\sum_{\sigma\in\Gamma_m}(-1)^{\ell(\sigma)}
(M_{\sigma(\lambda+\rho)-\rho}+
M_{\sigma(\lambda-\delta_m-\epsilon+\rho)-\rho}).
\end{equation}

By Lemma 4.1, we have that
\begin{equation}
[L_{\lambda-\delta_m-\epsilon}\otimes L_{\delta_1}: L_\mu]=1.
\end{equation}
Furthermore, since $\mu$ is typical, we can get from Lemma 4.2 that
\begin{equation}
[L_{\lambda-i(\delta_m+\epsilon)}\otimes L_{\delta_1}:
L_\mu]=[L_{\lambda-s(\delta_m+\epsilon)-\delta_m}\otimes
L_{\delta_1}: L_\mu]=0,\quad (i=2,3,\ldots,s).
\end{equation}

On the other hand, also thanks to that $\mu$ is typical, it is not
difficult to calculate that
\begin{equation}
[\sum_{\sigma\in\Gamma_m}(-1)^{\ell(\sigma)}(
M_{\sigma(\lambda+\rho)-\rho}+
M_{\sigma(\lambda-\delta_m-\epsilon+\rho)-\rho})\otimes
L_{\delta_1}:L_\mu]=2
\end{equation}

Combining (5.4)-(5.7), we get that
\begin{equation}
x_1\leq[L_\lambda\otimes L_{\delta_1}:L_\mu]+x_1=2.
\end{equation} Moreover, since
\begin{equation}
\dim\mbox{Hom}_\mathfrak{g}(L_\mu\otimes
L_{\delta_1},L_{\mu-\delta_m})=\dim\mbox{Hom}_\mathfrak{g}(L_\mu,L_{\mu-\delta_m}\otimes
L_{\delta_1})=1,
\end{equation} it should be that $x_1>0$.

Below we shall use induction on $s$ to obtain $L_\lambda$.

\noindent{\bf 1). For s=1.} Then by Theorem 3.1 ii), one can rewrite
(5.4) as
\begin{equation}
L_\lambda+x_1 L_{\lambda-\delta_m-\epsilon}+y
L_{\lambda-2\delta_m-\epsilon}=\sum_{\sigma\in\Gamma_m}(-1)^{\ell(\sigma)}
M_{\sigma(\lambda+\rho)-\rho}+L_{\lambda-\delta_m-\epsilon}-
L_{\lambda-2\delta_m-\epsilon}.
\end{equation} That is,
\begin{equation}
L_\lambda=\sum_{\sigma\in\Gamma_m}(-1)^{\ell(\sigma)}
M_{\sigma(\lambda+\rho)-\rho}+(1-x_1)L_{\lambda-\delta_m-\epsilon}-(1+y)
L_{\lambda-2\delta_m-\epsilon}.
\end{equation}

Multiply $L_{\delta_1}$ on the both sides of the above equation and
select the terms $L_\nu$ and $M_\nu$ with $\chi_\nu=\chi_\lambda$,
then we get
\begin{equation}
L_\lambda=\sum_{\sigma\in\Gamma_m}(-1)^{\ell(\sigma)}
M_{\sigma(\lambda+\rho)-\rho}+(1-x_1)L_{\lambda-2\delta_m-\epsilon}-(1+y)
L_{\lambda-\delta_m-\epsilon}.
\end{equation}
Comparing (5.11) with (5.12) shows us that
\begin{equation}
x_1-1=y+1\quad \Rightarrow \quad x_1=2+y\geq2.
\end{equation} Inequalities (5.8) and (5.13) induce that $x_1=2$ and
$y=0$. Thus
\begin{equation}
L_\lambda=\sum_{\sigma\in\Gamma_m}(-1)^{\ell(\sigma)}
M_{\sigma(\lambda+\rho)-\rho}-L_{\lambda-\delta_m-\epsilon}-
L_{\lambda-2\delta_m-\epsilon}.
\end{equation} As a co-product,
\begin{equation}
[L_\lambda\otimes
L_{\delta_1}:L_{\lambda-\delta_m}]=[L_\lambda\otimes
L_{\delta_1}:L_{\lambda-\epsilon}]=0.
\end{equation}

\noindent{\bf 2). For s=2.} Now using (5.14), equation (5.4) can be
rewritten as
\begin{eqnarray}
&&L_\lambda+x_1 L_{\lambda-\delta_m-\epsilon}+x_2
L_{\lambda-2\delta_m-2\epsilon}+y
L_{\lambda-3\delta_m-2\epsilon}\\\nonumber&&=\sum_{\sigma\in\Gamma_m}(-1)^{\ell(\sigma)}
M_{\sigma(\lambda+\rho)-\rho}+L_{\lambda-\delta_m-\epsilon}+
L_{\lambda-2\delta_m-2\epsilon}+L_{\lambda-3\delta_m-2\epsilon}.
\end{eqnarray}
That is,
\begin{equation}
L_\lambda=\sum_{\sigma\in\Gamma_m}(-1)^{\ell(\sigma)}
M_{\sigma(\lambda+\rho)-\rho}+(1-x_1)L_{\lambda-\delta_m-\epsilon}+
(1-x_2)L_{\lambda-2\delta_m-2\epsilon}+(1-y)L_{\lambda-3\delta_m-2\epsilon}.
\end{equation}
Notice that $\lambda-\delta_m-2\epsilon$ is a typical weight, hence
\begin{equation}
[L_\lambda\otimes L_{\delta_1}:
L_{\lambda-\delta_m-2\epsilon}]=[L_{(\lambda-2\delta_m-2\epsilon)^\delta}\otimes
L_{\delta_1}:L_{\lambda-\delta_m-2\epsilon}]=0
\end{equation} by Lemma 4.2,
\begin{equation}
[\sum_{\sigma\in\Gamma_m}(-1)^{\ell(\sigma)}
M_{\sigma(\lambda+\rho)-\rho}\otimes
L_{\delta_1}:L_{\lambda-\delta-2\epsilon_k}]=0
\end{equation} by direct calculation,
\begin{equation}
[L_{\lambda-\delta_m-\epsilon}\otimes
L_{\delta_1}:L_{\lambda-\delta_m-2\epsilon}]=0
\end{equation} by (5.15), and
\begin{equation}
[L_{\lambda-2\delta_m-2\epsilon}\otimes
L_{\delta_1}:L_{\lambda-\delta_m-2\epsilon}]=1
\end{equation} by Lemma 4.1.

Therefore (5.17)-(5.21) imply that $x_2=1$.

Multiply $L_\delta$ on the both sides of (5.17) and select the terms
$L_\nu$ and $M_\nu$ with $\chi_\nu=\chi_\lambda$, then we have
\begin{equation}
L_\lambda=\sum_{\sigma\in\Gamma_m}(-1)^{\ell(\sigma)}
M_{\sigma(\lambda+\rho)-\rho}+
(1-x_1)L_{\lambda-\delta_m-\epsilon}+
(1-x_2)L_{\lambda-3\delta_m-2\epsilon}+(1-y)L_{\lambda-2\delta_m-2\epsilon}.
\end{equation}
Compare (5.17) and (5.22), then we get
\begin{equation}
1-x_2=1-y\quad\Rightarrow\quad y=x_2=1.
\end{equation}
Therefore,
\begin{equation}
L_\lambda=\sum_{\sigma\in\Gamma_m}(-1)^{\ell(\sigma)}
M_{\sigma(\lambda+\rho)-\rho}+ (1-x_1)L_{\lambda-\delta_m-\epsilon}.
\end{equation}
By (5.8) and (5.9), we know that $x_1=1$ or $2$. But it is
impossible that $x_1=1$ because of Lemma 2.5. So $x_1=2$ and
\begin{equation}
L_\lambda=\sum_{\sigma\in\Gamma_m}(-1)^{\ell(\sigma)}
M_{\sigma(\lambda+\rho)-\rho}-L_{\lambda-\delta_m-\epsilon}.
\end{equation}
Now it is easy to get from the above equation that
\begin{equation}
[L_\lambda\otimes L_{\delta_1}:
L_{\lambda-\delta_m}]=[L_\lambda\otimes L_{\delta_1}:
L_{\lambda-\epsilon}]=0.
\end{equation}

\noindent{\bf 3). For any $\mbox{s}\geq\mbox{2}$.} We shall use
induction on $s$ to show that (5.25) and (5.26) hold for any
$s\geq2$. By induction assumption, we can get from (5.4) that
\begin{equation}
L_\lambda=\sum_{\sigma\in\Gamma_m}(-1)^{\ell(\sigma)}
M_{\sigma(\lambda+\rho)-\rho}+\sum_{i=1}^2(1-x_i)L_{\lambda-i\delta_m-i\epsilon}-
\sum_{i=3}^{s}x_iL_{\lambda-i\delta_m-i\epsilon}
-yL_{\lambda-(s+1)\delta_m-s\epsilon}.
\end{equation}

Multiply $L_{\delta_1}$ on both sides of the above formula. Then if
we choose the terms $M_\nu$ and $L_\nu$ with
$\chi_\nu=\chi_{\lambda-i\delta_m-(i+1)\epsilon} (1\leq i<s)$, there
comes that $x_2=1$, $x_3=\cdots=x_s=0$. Moreover, if we choose the
terms $M_\nu$ and $L_\nu$ with $\chi_\nu=\chi_\lambda$, there comes
that $y=x_s=0$.

Thus
\begin{equation}
L_\lambda=\sum_{\sigma\in\Gamma_m}(-1)^{\ell(\sigma)}
M_{\sigma(\lambda+\rho)-\rho}+(1-x_1)L_{\lambda-\delta_m-\epsilon}.
\end{equation}
Recall that we have known $x_1=1$ or $2$. But only the case of
$x_1=2$ is valid by (5.28) because of Lemma 2.5. Hence we have shown
that (5.25) and (5.26) hold for any $s\geq2$.

\subsubsection{\bf In case that $\lambda$ is $(\delta_k+\epsilon)$-atypical,
$(k<m)$.} In this case,
$\lambda_0+1+k-m=\lambda_k>1$ and $\lambda_k\geq\lambda_{k+1}\geq\lambda_{k+2}\geq\cdots\geq\lambda_m\geq1$.

Let
\begin{equation}
\lambda^{(t)}=\sum_{i=1}^k\lambda_i\delta_i+\sum_{i=k+1}^{m}\lambda_k\delta_i+(\lambda_k-1+t)\epsilon,\quad(t=0,1\ldots,m-k),
\end{equation}
\begin{equation*}
\lambda^{(m-k+t)}=\sum_{i=1}^k\lambda_i\delta_i+
\sum_{i=k+1}^{m}\lambda_k\delta_i+\lambda_0\epsilon-t\delta_m,\quad(t=1,2,\ldots,\lambda_k-\lambda_m),
\end{equation*}
\begin{equation*}
\lambda^{(m-k+\lambda_k-\lambda_m+t)}
=\sum_{i=1}^k\lambda_i\delta_i
+\sum_{i=k+1}^{m-1}\lambda_k\delta_i+\lambda_m\delta_m+\lambda_0\epsilon-t\delta_{m-1},\quad(t=1,2,\ldots,\lambda_k-\lambda_{m-1}),
\end{equation*}
\begin{equation*}
\cdots\cdots\cdots\cdots\cdots\cdots
\end{equation*}
\begin{equation*}
\lambda^{(m-k+\sum_{i=k+2}^m(\lambda_k-\lambda_i)+t)}=\sum_{i=1}^k\lambda_i\delta_i
+\lambda_k\delta_{k+1}+\sum_{i=k+1}^{m}\lambda_i\delta_i+\lambda_0\epsilon-t\delta_{k+1},\quad(t=1,2,\ldots,\lambda_k-\lambda_{k+1}),
\end{equation*}
\begin{equation*}
\lambda^{(m-k+\sum_{i=k+1}^m(\lambda_k-\lambda_i))}=\lambda.
\end{equation*}

Notice that all $\lambda^{(t)}$ $(0\leq t\leq
m-k+\sum_{i=k+1}^m(\lambda_k-\lambda_i))$ are atypical weights and
$\lambda^{(t+1)}\in\mathcal {P}_{\lambda^{(t)}}$.

We have obtained in Section 5.2.1 that
\begin{equation}
L_{\lambda^{(0)}}=\sum_{\sigma\in\Gamma_m}(-1)^{\ell(\sigma)}
M_{\sigma(\lambda^{(0)}+\rho)-\rho}-L_{\varphi(\lambda^{(0)})}.
\end{equation}

Using the algorithm introduced in Remark 4.5, now there is no
difficulty for us to calculate that
\begin{equation}
L_{\lambda}=\sum_{\sigma\in\Gamma_m}(-1)^{\ell(\sigma)}
M_{\sigma(\lambda+\rho)-\rho}-L_{\varphi(\lambda)}
\end{equation}
by (5.29) and (5.30).

\subsubsection{\bf Final results for the case $\varphi(\lambda)\neq\lambda^T$}
\

To summarize:
\begin{theorem}
For any non-tail atypical weight $\lambda$ with
$\varphi(\lambda)\neq\lambda^T$,
\begin{equation}
L_\lambda=\left\{\begin{array}{ll}\sum_{\sigma\in\Gamma_m}(-1)^{\ell(\sigma)}
M_{\sigma(\lambda+\rho)-\rho}-L_{\varphi(\lambda)},&\quad\mbox{if
$\varphi^2(\lambda)\neq\lambda^T$}\\\sum_{\sigma\in\Gamma_m}(-1)^{\ell(\sigma)}
M_{\sigma(\lambda+\rho)-\rho}-L_{\varphi(\lambda)}-L_{\lambda^T},&\quad\mbox{if
$\varphi^2(\lambda)=\lambda^T$}.
\end{array}\right.
\end{equation}
\end{theorem}

\vspace{0.3cm}
For any atypical weight $\lambda\in\mathcal {P}$, denote by $\theta_\lambda$ the unique number
such that $\varphi^{\theta_\lambda+1}(\lambda)=\lambda^T$.

\begin{theorem} For any atypical weight $\lambda\in\mathcal {P}$ with $\theta_\lambda\geq1$, one has
\begin{eqnarray}
L_\lambda=\sum_{i=0}^{\theta_\lambda}\sum_{\sigma\in\Gamma_m}(-1)^{i+\ell(\sigma)}
M_{\sigma(\varphi^i(\lambda)+\rho)-\rho}+\quad\quad\quad\quad\quad\quad\quad\quad\quad\quad\quad
\\\nonumber
\sum_{\sigma\in \Gamma_{m-1}}\sum_{i=\flat(\sigma,\lambda^T)}^m
\sum_{j=\mbox{max}\{0,\frac{1}{2}-\sigma(\lambda^T+\rho)_{i-1}\}}^{-\frac{3}{2}-\tau_i\sigma(\lambda^T+\rho)_{i+1}}
2(-1)^{\theta_\lambda+\ell(\tau_i\sigma)+j}M_{\tau_i\sigma(\lambda^T+\rho)-\rho+j\epsilon-j\delta_i}.
\end{eqnarray}
\end{theorem}

\bibliographystyle{amsplain}

\end{document}